\documentclass[a4paper, 10pt]{article}

\usepackage[english]{babel}
\usepackage{enumerate}
\usepackage{bbm}
\usepackage[all]{xy}
\usepackage{amsthm}
\usepackage{amsmath}

\newcommand{\Arr}{\ensuremath{\Rightarrow}}

\newcommand{\cateb}[1]{\ensuremath{\mathbbm{#1}}}

\newcommand{\arr}{\ensuremath{\rightarrow}}
\newcommand{\hook}{\ensuremath{\hookrightarrow}}
\newcommand{\rarr}{\ensuremath{\rightarrow}}
\newcommand{\slice}[2]{\ensuremath{#1/{\textstyle#2}}}
\newcommand{\oslice}[2]{\ensuremath{{\textstyle#1}/#2}}

\newcommand{\bbo}[1]{\ensuremath{\mathbbm{#1}}}

\newcommand{\fib}[3]{\ensuremath{#3:\bbo#1\rarr\bbo#2}}

\newtheorem{proposition}{Proposition}

\newtheorem{example}{Example}
\newtheorem{examples}{Examples}
\newtheorem{definition}{Definition}
\newtheorem{remark}{Remark}

\begin{document}
\title{{\bf A note on the slicing of fibrations}}
\date{}
\author{Ruggero Pagnan}
\maketitle
\begin{abstract}
We describe the construction of the slice fibration of a given one. 
\end{abstract}

\tableofcontents

\section{Introduction}
In this note we describe the construction of the slice fibration of a given one. Although maybe an already well known construction, we have not been able to find it in any of the more or less standard references for fibered category theory, hence the reason for explicitly considering it, also in the spirit of~\cite{MR780520}. 

\section{Categorical preliminaries}\label{prelim}
We here recall some fundamental category theoretic facts regarding ordinary categories and fibered categories, or fibrations, in the sense of~\cite{Grothendieck}, mainly to establish the terminology and the notations that will be most frequently used henceforth. We assume that the reader is already acquainted with the general basics of ordinary category theory and fibered category theory, but see~\cite{Streicher},~\cite{MR1712872},~\cite{MR1182992},~\cite{MR780520} and~\cite{MR1674451} to find the notions that we will not explicitly reintroduce in this note.

\subsection{Slice categories and functors between them}\label{onslicecats}
For \cateb{C} a category and $A$ any of its objects, the \emph{slice category of \cateb{C} over $A$}, usually denoted \slice{\cateb C}{A}, has objects given by morphisms $x:X\arr A$ of \cateb C, whereas a morphism of \slice{\cateb C}{A} say from $x:X\arr A$ to $y:Y\arr A$, is a moprhism $g:X\arr Y$ of \cateb C such that $yg=x$ in \cateb C or, more briefly, a commutative triangle
$$
\xymatrix{X\ar[dr]_x\ar[rr]^g&&Y\ar[dl]^y\\
&A}
$$ 
thus making explicit that \slice{\cateb C}{A} is a subcategory of $\cateb{C}^{\arr}$, that is the so-called category of morphisms of $\cateb C$. 
\begin{proposition}\label{fewprop}
For \cateb C a category, the following facts hold:
\begin{enumerate}[(i)]
\item\label{i} If \cateb C is equipped with a terminal object $1$, then $\slice{\cateb C}{1}\equiv\cateb C$.

\item For every morphism $f:A\arr B$ of \cateb C,   $\slice{(\slice{\cateb C}{B})}{f}\equiv\slice{\cateb C}{A}$.
\item\label{iii} For every morphism $f:A\arr B$ of \cateb C, there is the  functor 
$$\Sigma_f:\slice{\cateb C}{A}\arr\slice{\cateb C}{B}$$
identified by the assignment
$$
\xymatrix{X\ar[dr]_x\ar[rr]^g&&Y\ar[dl]^y\\
&A}\qquad\mapsto\qquad
\xymatrix{X\ar[dr]_{fx}\ar[rr]^g&&Y\ar[dl]^{fy}\\
&B}
$$
\item\label{iv} For every object $A$ of a category \cateb C, there is the ``obvious'' forgetful functor $\Sigma_A:\slice{\cateb C}{A}\arr\cateb C$. 
\item For every morphism $f:A\arr B$, if \cateb C is equipped with pullbacks along $f$, then there is the adjoint situation
$$
\xymatrix{\slice{\cateb C}{A}\ar@{}[rr]|{\perp}\ar@/^/[rr]^{\Sigma_f}&&\slice{\cateb C}{B}\ar@/^/[ll]^{f^*}}
$$
\item For every object $A$ of \cateb C, the category \slice{\cateb C}{A} is equipped with a canonical terminal object, that is $1_A$ 
\end{enumerate}
\end{proposition}
\begin{proof}
Straightforward.
\end{proof}

\begin{remark}\label{riferimento}
If \cateb C is equipped with a terminal object $1$, then the functor $\Sigma_A$ in~\eqref{iv} arises as a special case of \eqref{iii}, if, by abusing notation, $A$ is $A:A\arr 1$.
\end{remark}

\subsection{Pointed categories}
Suitably, taking slices of categories amounts to a universal construction. A \emph{pointed category} is a pair $(A,\cateb C)$ where \cateb C is a non-empty category and $A$ is an object of \cateb C, that is a functor $A:\cateb 1\arr\cateb C$. A category \cateb C equipped with a (chosen) terminal object $1$ identifies a pointed category $(1,\cateb C)$, which we call \emph{terminally pointed} to be more specific. A \emph{morphism of pointed categories}, say from $(A, \cateb C)$ to $(B, \cateb D)$ is a pair $(F, \alpha)$ with $F:\cateb C\arr\cateb D$ a functor and $\alpha:FA\arr B$ a morphism of \cateb D, that is a natural transformation $\alpha:F\circ A\Arr B$. 
\begin{remark}
A morphism of pointed categories whose codomain is a terminally pointed category, say $(F,\alpha):(A,\cateb C)\arr(1,\cateb D)$ amounts to a functor $F:\cateb C\arr\cateb D$ together with the unique morphism $FA:FA\arr 1$ in \cateb D, which essentially is the object $FA$ of \cateb D. 
\end{remark}
A pair of composable morphisms of pointed categories compose in accordance with the following law of composition:
$$
\xymatrix{(A,\cateb D)\ar@/_1pc/[rr]_{(GF,\beta\circ(G\alpha))}\ar[r]^{(F,\alpha)}&(B,\cateb C)\ar[r]^{(G,\beta)}&(C,\cateb B)}
$$
Moreover, for every pointed category $(A,\cateb D)$, $(1_{\cateb D},1_A):(A,\cateb D)\arr(A,\cateb D)$ is an identity with respect to the previous law of composition. Thus, pointed categories and morphisms between them identify a category ${\bf PtdCat}$; in turn, terminally pointed categories identify a full subcategory $\iota:{\bf TPtdCat}\hook{\bf PtdCat}$.
\vfill
For a functor $F:\cateb C\arr\cateb D$ and an object $A$ of \cateb C there is the \emph{slice functor} 
$\slice F A:\slice{\cateb C}{A}\arr\slice{\cateb D}{FA}$ whose action on objects and morphisms is completely identified by the assignment 
$$
\xymatrix{X\ar[dr]_x\ar[rr]^g&&Y\ar[dl]^y\\
&A}\qquad\mapsto\qquad
\xymatrix{FX\ar[dr]_{Fx}\ar[rr]^{Fg}&&FY\ar[dl]^{Fy}\\
&FA}
$$
As a consequence of this, there is the functor $G:{\bf PtdCat}\arr{\bf TPtdCat}$ completely identified by the assignment
$$
\xymatrix{(A,\cateb C)\ar[d]_{(F,\alpha)}\\
(B,\cateb D)}
\quad\mapsto\quad
\xymatrix{(1_A,\slice{\cateb C}{A})\ar[d]^{(\Sigma_{\alpha} \circ\slice{\displaystyle F}{A},\alpha)}\\
(1_B,\slice{\cateb D}{B})}
$$
where
$$
\xymatrix{\slice{\cateb C}{A}\ar@/^1pc/[rr]^{\Sigma_{\alpha}\circ\slice{\displaystyle F}{A}}\ar[r]_{\slice{\displaystyle F}{A}}&\slice{\cateb D}{FA}\ar[r]_{\Sigma_{\alpha}}&\slice{\cateb D}{B}}
$$
\begin{proposition}\label{firstobs}
There is an adjoint situation
$$
\xymatrix{{\bf TPtdCat}\ar@{^(->}[rr]_{\iota}^{\top}&&{\bf PtdCat}\ar@/_1pc/[ll]_G}
$$
\end{proposition}
\begin{proof}
For every pointed category $(A,\cateb C)$, the assignment $(A,\cateb C)\mapsto(1_A,\slice{\cateb C}{A})$ extends to the previously identified functor $G$ because 
$$
(\Sigma_A,1_A):(1_A,\slice{\cateb C}{A})\arr(A,\cateb C)
$$ 
is universal from $\iota$ to $(A,\cateb C)$. In fact , for every terminally pointed category $(1,\cateb D)$ and moprhism $(F,\alpha):(1,\cateb D)\arr(A,\cateb C)$, there is a unique morphism of (terminally) pointed categories $(\overline F,\overline\alpha):(1,\cateb D)\arr(1_A,\slice{\cateb C}{A})$ making the diagram
$$
\xymatrix{(1_A,\slice{\cateb C}{A})\ar[rr]^{(\Sigma_A,1_A)}&&(A,\cateb C)\\
(1,\cateb D)\ar[u]^{(\overline F,\overline\alpha)}\ar@/_/[urr]_{(F,\alpha)}}
$$
commute in {\bf PtdCat},
where
$$
\overline F=\Sigma_{\alpha}\circ\slice{\displaystyle F}{1}:\xymatrix{\cateb D\ar[r]^(.45){\slice{\displaystyle F}{1}}&\slice{\cateb C}{F1}\ar[r]^{\Sigma_{\alpha}}&\slice{\cateb C}{A}}
$$ 
and $\overline\alpha = \alpha$ as the unique morphism from $\alpha$ to $1_A$ in \slice{\cateb{C}}{A}. 
\end{proof}

\subsection{Fundamentals of fibered category theory}  
Let \fib{X}{B}{P} be a functor and let $f:X\arr Y$ be a morphism of \cateb X. The morphism $f$ is $P$-\emph{cartesian} if for every morphism $v:K\arr PX$ of \cateb B, for every morphism $g:Z\arr Y$ of \cateb X such that $Pg=Pf\circ v$, there exists a unique morphism $h:Z\arr X$ such that $f\circ h=g$ and $Ph=v$, as in the diagram
$$
\xymatrix{Z\ar@/^/[drr]^g\ar@{-->}[dr]_h\\
&X\ar[r]_f&Y\\
K\ar[r]_{v}&PX\ar[r]_{Pf}&PY}
$$
$P$ is a \emph{fibered category} or a \emph{(Grothendieck) fibration} if for every object $Y$ of \cateb X, for every morphism $u:I\arr PY$ of \cateb B, there exists a $P$-cartesian morphism $f:X\arr Y$ such that $Pf=u$. If this is the case, then $f$ is also referred to as a $P$-\emph{cartesian lifting} of $Y$ along $u$, and $P$ is also said to have \emph{enough cartesian liftings}. For \fib{X}{B}{P} a fibration, a morphism $f:X\arr Y$ of \cateb X is $P$-\emph{vertical} if $Pf=1_{PX}$; it is immediate to verify that cartesian liftings in a fibration a determined up to a unique vertical isomorphism. For every object $I$ of \cateb B, the $P$-vertical morphisms whose image is $1_I$ identify a subcategory $P_I\hook\cateb X$ which is the \emph{fiber category of $P$ over $I$}; actually, $P_I$ can be obtained as in the change of base situation
$$
\xymatrix{P_I\ar[d]\ar[rr]&&\cateb X\ar[d]^P\\
\cateb 1\ar[rr]_I&&\cateb B}
$$
The category which is the domain of a fibration is referred to as its \emph{total category}; the category which is the codomain of a fibration is referred to as its \emph{base category}.
\begin{example}
For \cateb B a category, the codomain functor ${\bf cod}_{\cateb B}:\cateb{B}^{\arr}\arr\cateb B$ is a fibration if and only if \cateb B has pullbacks. 
\end{example}
\begin{example}
For \cateb B a category, the domain functor ${\bf dom}_{\cateb B}:\cateb{B}^{\arr}\arr\cateb B$ is a fibration. For every object $y:Y\arr J$ of $\cateb{B}^{\arr}$, for every morphism $f:X\arr Y$ of \cateb B, a cartesian lifting of $y$  along $f$ is 
$$
\xymatrix{X\ar[d]_{yf}\ar[r]^f&Y\ar[d]^y\\
J\ar[r]_{1_J}&J}
$$
and if $I$ is an object of \cateb B, then the fiber category of ${\bf dom}_{\cateb B}$ over $I$ is the opslice category \oslice{I}{\cateb B}.
\end{example}
\begin{example}
For \cateb X a category, the functor $\pi_1:\cateb B\times\cateb X\arr\cateb B$ is a fibration. Every fiber category of $\pi_1$ is equivalent to \cateb X.
\end{example}
\begin{example}\label{identity}
For every category \cateb B, the identity functor $1_{\cateb B}:\cateb B\arr\cateb B$ is a fibration. Every morphism of \cateb B is $1_{\cateb B}$-cartesian. Every fiber category of $1_{\cateb B}$ is equivalent to the terminal category \cateb 1.
\end{example}
\begin{example}
For every category \cateb C, the functor $\cateb C:\cateb C\arr\cateb 1$ is a fibration. \cateb C-cartesian morphisms are exactly the isomorphisms of \cateb C. 
\end{example}
\begin{example}\label{six}
Let \fib{X}{B}{P} be a fibration. The $P$-vertical morphisms identify a full subcategory $V(P)\hook\cateb{X}^{\arr}$. The codomain functor ${\bf cod}_P:V(P)\arr\cateb X$ is a fibration if and only if $P$ has fibered pullbacks; that is, every fiber category of $P$ has pullbacks which are stable by $P$-reindexing.
\end{example}

\subsection{Fibered functors and fibered adjuctions}
Let \fib{X}{B}{P} and \fib{Y}{B}{Q} be fibrations. A \emph{fibered functor} from $P$ to $Q$ is a functor $F:\cateb X\arr\cateb Y$ mapping $P$-cartesian morphisms to $Q$-cartesian morphisms and such that $Q\circ F=P$. When confusion is not likely to arise, such a fibered functor will be also more briefly written $F:P\arr_{\cateb B}Q$. For a pair of parallel fibered functors $F, G:P\arr_{\cateb B}Q$, a \emph{fibered natural transformation} from $F$ to $G$ is a natural transformation from $F$ to $G$ with $Q$-vertical components.
Fibrations with base category \cateb B, fibered functors and fibered natural transformations form a $2$-category ${\bf Fib}(\cateb B)$. The fibration $1_{\cateb B}$ (see example~\ref{identity}) is a terminal object in ${\bf Fib}(\cateb B)$.
For fibrations \fib XBP and \fib YBQ, a {\em fibered adjunction} between $P$ and $Q$ in {\bf Fib}(\bbo B) is an adjunction
$$
\xymatrix{\bbo X\ar@/^/[rr]^F\ar@{}[rr]|{\perp}&&\bbo Y\ar@/^/[ll]^G}
$$ 
with $P$-vertical unit, $Q$-vertical counit and $F:P\rarr_{\cateb B} Q$, $G:Q\rarr_{\cateb B} P$ fibered functors.

\section{Slice fibrations}
In this section we describe the construction of the slice fibration of a given one; its fiber categories turn out to be the slice categories of the fibers of the original fibration. Moreover, the construction at issue is suitably universal.

\subsection{Pointed fibrations}
\begin{definition}
Let \cateb B be a category. A \emph{pointed fibration (over \cateb B)} is a pair $(p,P)$ where \fib XBP is a fibration and $p:1_{\cateb B}\arr_{\cateb B} P$ is a fibered functor. A \emph{terminally pointed fibration (over \cateb B)} is a pointed fibration $(p,\fib XBP)$ where $p$ is fibered right adjoint to $P:P\arr_{\cateb B}1_{\cateb B}$. The first component of a pointed fibration $(p,P)$ will be also referred to as its \emph{point} and, when confusion is not likely to arise, we will  also say that a fibration \fib XBP is pointed, without explicitly providing the pertinent point implicitly referred to.
\end{definition}

\begin{remark}\label{firstpointed}
If $(p, \fib XBP)$ is a pointed fibration over \cateb B, then for every morphism $u:I\arr J$ of \cateb B, the morphism $pu:pI\arr pJ$ of \cateb X is $P$-cartesian, because every morphism of \cateb B is $1_{\cateb B}$-cartesian and $p$ is required to be fibered.
\end{remark}

\begin{remark}
A terminally pointed fibration is nothing but a fibration equipped with fibered terminal objects; thus, a terminally pointed fibration will be henceforth also written as a pair $(\top_{\bullet},P)$.
\end{remark}

\begin{definition}\label{ptdmorphism}
A \emph{morphism of pointed fibrations (over \cateb B)}, say from $(p,P)$ to $(q, Q)$ is a pair $(F, \alpha)$ where $F:P\arr_{\cateb B}Q$ is a fibered functor
and $\alpha:F\circ p\Arr q$ is a natural transformation with $Q$-vertical components. 
Morphisms of pointed fibrations over \cateb B compose like this:
$$
\xymatrix{(p,P)\ar@/_1pc/[rr]_{(GF,\beta\circ G\alpha)}\ar[r]^{(F,\alpha)}&(q, Q)\ar[r]^{(G,\beta)}&(r,R)}
$$
For a pointed fibration $(p,\fib XBP)$, $(1_{\cateb X}, 1_p):(p,P)\arr(p,P)$ is an identity with respect to the previous law of composition. 
\end{definition}
\begin{remark}
As a consequence of remark~\ref{firstpointed}, if $(F,\alpha):(p,P)\arr(q, Q)$ is a morphism of pointed fibrations over \cateb B, then for every morphism $u:I\arr J$ of \cateb B, the commutative diagram
$$
\xymatrix{FpI\ar[d]_{\alpha_I}\ar[rr]^{Fpu}&&FpJ\ar[d]^{\alpha_J}\\
qI\ar[rr]_{qu}&&qJ}
$$
is a pullback in the total category of $Q$.
\end{remark}

\begin{remark}
For a parallel pair of morphisms of pointed fibrations over \cateb B, say $(F,\alpha), (G,\beta):(p,\fib XBP)\arr(q,\fib YBQ)$, a $2$-cell $\gamma:(F, \alpha)\Arr(G, \beta)$ is a fibered natural transformation $\gamma:F\Arr G$ such that $\beta\circ \gamma p=\alpha$, that is: for every object $I$ of \cateb B $\beta_I\circ\gamma_{pI}=\alpha_I$.
\end{remark}
Forgetting $2$-cells, pointed fibrations over \cateb B as objects and the morphisms between them
identify a category ${\bf PtdFib}(\cateb B)$. Clearly, the terminally pointed fibrations over \cateb B identify a full subcategory ${\bf TPtdFib}(\cateb B)\hook{\bf PtdFib}(\cateb B)$.

\begin{proposition}
Pointed fibrations are stable under change of base: for a pointed fibration $(p,\fib XBP)$ and a functor $F:\cateb C\arr\cateb B$, the fibration $P_F$ arising by change of base of $P$ along $F$ is a pointed. Similarly, terminally pointed fibrations are stable under change of base.
\end{proposition}
\begin{proof}
It suffices to have a look at the diagram
$$
\xymatrix{&\cateb C\ar@/^/[ddl]|(.45){\hole}^(.65){1_{\cateb C}}\ar@{-->}[dl]_{p_F}\ar[rr]^F&&\cateb B\ar@/^/[ddl]^{1_{\cateb B}}\ar[dl]_p\\
\cateb Y\ar[d]_{P_F}\ar[rr]^{\overline F}&&\cateb X\ar[d]_P\\
\cateb C\ar[rr]_F&&\cateb B}
$$
in which all the quadrilaterals are pullbacks and the point $p_F$ is uniquely induced. Explicitly, the action of $p_F$ on an object $X$ of \cateb C is identified by the assignment $X\mapsto (X, pFX)$, which extends to the morphisms of \cateb C in accordance with the assignment 
$$f:X\arr Y\quad\mapsto\quad (f,pFf):(X,pFX)\arr(Y,pFY)$$
\end{proof}

\subsection{Slice fibrations}

\begin{definition}\label{fundef}
Let $(p,\fib XBP)$ be a pointed fibration where $P$ is equipped with fibered pullbacks. The \emph{slice fibration of $P$ over $p$} is the fibration $\slice P p:\cateb Y\arr\cateb B$ obtained in the change of base situation
\begin{eqnarray}\label{slicefib}
\xymatrix{\cateb Y\ar[d]_{\slice P p}\ar[rr]^{\overline p}&&V(P)\ar[d]^{\bf{cod}_P}\\
\cateb B\ar[rr]_p&&\cateb X}
\end{eqnarray}
\end{definition}

\begin{remark}
For a fibration \fib XBP with fibered pullbacks the composite $P\circ{\bf cod}_P$ is a fibration which is usually denoted $P^{\arr}$ and referred to as the \emph{fibration of morphisms} associated to $P$; the fibration ${\bf cod}_P$ (see example~\eqref{six}) is a fibered functor ${\bf cod}_P:P^{\arr}\arr_{\cateb B} P$.
\end{remark}
With reference to diagram~\eqref{slicefib}, \cateb Y is the category whose objects are pairs $(I, x:X\arr pI)$ where $I$ is an object of \cateb B and $x:X\arr pI$ is a $P$-vertical morphism over $I$; a morphism of \cateb Y, say from $(I, x:X\arr pI)$ to $(J,y:Y\arr pJ)$ is a pair $(u,f)$ with $u:I\arr J$ a morphism of \cateb B and $f:X\arr Y$ a morphism of \cateb X, making the diagram
\begin{eqnarray}\label{codcart}
\xymatrix{X\ar[d]_x\ar[rr]^f&&Y\ar[d]^y\\
pI\ar[rr]_{pu}&&pJ}
\end{eqnarray}
commute in \cateb X. The fibration $\slice P p$ acts on both objects and morphisms of $\cateb Y$ by projecting their first component to \cateb B. A $\slice P p$-cartesian morphism over a morphism $u:I\arr J$ of \cateb B is a ${\bf cod}_P$-cartesian morphism over $pu$, that is a commutative diagram such as~\eqref{codcart} which is a pullback. 

\begin{remark}
For every object $I$ of \cateb B, the fiber category of $\slice P p$ over $I$ is, essentially, the slice category $\slice{P_I}{pI}$, that is $(\slice P p)_I\equiv\slice{P_I}{pI}$.
\end{remark}

\begin{proposition}
A slice fibration $\slice P p:\cateb Y\arr\cateb B$ is a terminally pointed fibration, namely: it is a fibration equipped with fibered terminal objects. 
\end{proposition}
\begin{proof}
For every object $I$ of \cateb B, the assignment $I\mapsto (I,1_{pI})$ extends to a fibered functor $1_{p\bullet}:1_{\cateb B}\arr_{\cateb B}\slice P p$ which is right adjoint right inverse to $\slice P p$.
\end{proof}

\begin{examples}
\indent
\begin{enumerate}[i)]
\item For a terminally pointed fibration $(\top_{\bullet}, \fib XBP)$ with $P$ equipped with fibered pullbacks, it is well known that $P$ can be recovered by change of base, as follows:
$$
\xymatrix{\cateb X\ar[d]_P\ar[r]&V(P)\ar[d]^{\bf{cod}_P}\\
\cateb B\ar[r]_{\top_{\bullet}}&\cateb X}
$$ 
This fact is the counterpart of point \eqref{i} of proposition~\ref{fewprop}; explicitly, here we have $\slice P \top_{\bullet}\equiv P$.
\item For \cateb C a category and $A$ any of its objects, the pointed category $(A,\cateb C)$ is nothing but the pointed fibration $(A:\cateb 1\arr\cateb C,\cateb C:\cateb C\arr\cateb 1)$ where the fibration \cateb C is equipped with fibered pullbacks if and only if the ordinary category \cateb C is equipped with pullbacks if and only if ${\bf cod}_{\cateb C}:V(\cateb C)\arr\cateb C$ is a fibration, where $V(\cateb C)\equiv\cateb{C}^{\arr}$. Thus, as a special case of the construction introduced in definition~\ref{fundef}, the ordinary slice category \slice{\cateb C}{A} can be recovered as a slice fibration by change of base, as follows:
$$
\xymatrix{\slice{\cateb C}{A}\ar[d]_{\slice{\cateb C}{A}}\ar[r]&\cateb{C}^{\arr}\ar[d]^{{\bf cod}_{\cateb C}}\\
\cateb 1\ar[r]_A&\cateb C}
$$
\item For \cateb C a category, the diagonal functor $\Delta_{\cateb C}=<1_{\cateb C},1_{\cateb C}>:\cateb C\arr\cateb C\times\cateb C$ is a point of the projection fibration $\pi_1:\cateb C\times\cateb C\arr\cateb C$, so that $(\Delta_{\cateb C},\pi_1)$ is a pointed fibration; moreover, $\pi_1$ is equipped with fibered pullbacks if and only if \cateb C is equipped with pullbacks. If this is the case, the fundamentale fibration ${\bf cod}_{\cateb C}:\cateb{C}^{\arr}\arr\cateb C$ arises by change of base, as follows:
$$
\xymatrix{\cateb{C}^{\arr}\ar[d]_{{\bf cod}_{\cateb C}}\ar[r]&V(\pi_1)\ar[d]^{{\bf cod}_{\pi_1}}\\
\cateb C\ar[r]_{\Delta_{\cateb C}}&\cateb C\times\cateb C}
$$
that is ${\bf cod}_{\cateb C}\equiv\slice{\pi_1}{\Delta_{\cateb C}}$.
\end{enumerate}
\end{examples}

\subsection{Fibered functors between slice fibrations}
A morphism of pointed fibrations $(F,\alpha):(p,P)\arr(q,P)$ with $F$ the identity, is completely identified by a fibered natural transformation $\alpha:p\Arr q$; if this is the case and if confusion is not likely to arise, we will just write $\alpha:(p,P)\arr(q,P)$.

\begin{proposition}\label{prima}
For every morphism of pointed fibrations $\alpha:(p,P)\arr(q,P)$ the following facts hold:
\begin{enumerate}[(i)]
\item for every object $(I, x:X\arr qI)$ in the toal category of $\slice P q$, the pullback diagram 
$$
\xymatrix{X\ar[d]_x&X'\ar[l]\ar[d]^{x'}\\
qI&pI\ar[l]^{\alpha_I}}
$$
identifies the assignment 
$$
(I, x:X\arr qI)\mapsto (I, x':X'\arr qI)
$$
which extends to a fibered functor $\alpha^{\ast}:\slice{P}{q}\arr_{\cateb B}\slice{P}{p}$.
\item\label{II} for every object $(I, x:X\arr pI)$ in the total category of $\slice P p$, the assignment
$$
(I, x:X\arr pI)\mapsto(I, \alpha_I\circ x:X\arr qI)
$$
extends to a fibered functor $\Sigma_{\alpha}:\slice P p\arr\slice P q$.
\item There is a fibered adjunction 
$$
\xymatrix{\slice P p \ar@{}[rr]|{\perp}\ar@/^2ex/[rr]^{\Sigma_{\alpha}}&&\slice P q\ar@/^2ex/[ll]^{\alpha^{\ast}}}
$$
in ${\bf Fib}(\cateb B)$
\end{enumerate}
\end{proposition}
\begin{proof}
\begin{enumerate}[(i)]
\item Straightforward; in particular, $\alpha^{\ast}$ maps $\slice P q$-cartesian morphisms to $\slice P p$-cartesian morphisms by pullback pasting because a $\slice P q$-cartesian morphism such as 
$$
(u,f):(I,x:X\arr qI)\arr(J,y:Y\arr qJ)
$$ 
essentially amounts to the pullback diagram 
$$
\xymatrix{X\ar[d]_x\ar[rr]^f&&Y\ar[d]^y\\
qI\ar[rr]_{qu}&&qJ}
$$
\item For every morphism $(u,f):(I,x:X\arr pI)\arr(J,y:Y\arr qJ)$ in the total category of $\slice P p$, $\Sigma_{\alpha}(u,f)$ is 
$$
(u,f):(I,\alpha_I\circ x:X\arr qI)\arr(J,\alpha_J\circ y:Y\arr qJ)
$$
in accordance with the commutative diagram
$$
\xymatrix{Y\ar[d]_y&X\ar[l]_f\ar[d]^x\\
pJ\ar[d]_{\alpha_J}&pI\ar[l]_{pu}\ar[d]^{\alpha_I}\\
qJ&qI\ar[l]^{qu}}
$$ 
and $\Sigma_{\alpha}$ preserves $\slice P p$ cartesian morphisms because the lower square in the previous diagram is a pullback, thanks to the $P$-cartesianness of $pu$ and $qu$. 
\item Straightforward.
\end{enumerate}
\end{proof}

\begin{remark}
In parallel with point~\eqref{iv} of proposition~\ref{fewprop} and remark~\ref{riferimento}, we observe that for every pointed fibration $(p,P)$ there exists an abvious fibered forgetful functor $\Sigma_p:\slice P p\arr_{\cateb B}P$; this is because in a $\slice P p$-cartesian morphism such as~\eqref{codcart}, the morphism $f$ turns out to be  $P$-cartesian. Moreover, if $P$ as a pointed fibration $(p, P)$ is also terminally pointed as $(\top_{\bullet},P)$, then the forgetful fibered functor $\Sigma_p$ really amounts to $\Sigma_p:\slice P p\arr_{\cateb B}\slice{P}{\top_{\bullet}}\equiv P$ in accordance with point~\eqref{II} of proposition~\ref{prima}.
\end{remark}

\subsection{Slicing fibrations as a universal construction}
In parallel with proposition~\ref{firstobs} we provide
\begin{proposition}
There is an adjoint situation
\begin{eqnarray}\label{fiberwise}
\xymatrix{{\bf TPtdFib}(\cateb B)\ar@{^(->}[rr]_{\iota}^{\top}&&{\bf PtdFib}(\cateb B)\ar@/_1pc/[ll]_G}
\end{eqnarray}
\end{proposition}
\begin{proof}
For every pointed fibration $(q,Q)$, the assignment $(q,Q)\mapsto (1_{q\bullet},\slice Q q)$ extends to a functor $G:{\bf Ptd}(\cateb B)\arr{\bf TPtdFib}(\cateb B)$ which is right adjoint to the inclusion functor $\iota$, because for every pointed fibration $(q, Q)$ the morphism 
$$
(\Sigma_q,1_q):(1_{q\bullet},\slice Q q)\arr(q, Q)
$$ 
is universal from $\iota$ to $(q,Q)$. 
\end{proof}

For any pair of pointed fibrations over possibly different base categories, say $(p, \fib XBP)$ and $(q,\fib YCQ)$,  a morphism from the first to the second is a pair $((H,K),\alpha)$ with $(H,K)$ a morphism in {\bf Fib} from $P$ to $Q$ and $\alpha:H\circ p\Arr q\circ K$ with $Q$-vertical components. Pointed fibrations and the morphisms just described between them form a category {\bf PtdFib}. Analogously, the terminally pointed fibrations form a full subcategory {\bf TPtdFib} of {\bf PtdFib}.  

\begin{proposition}
The adjoint situation~\eqref{fiberwise} extends to a fibered adjunction
$$
\xymatrix{{\bf TPtdFib}\ar[dr]\ar@{^(->}[rr]_{\iota}^{\top}&&{\bf PtdFib}\ar@/_1pc/[ll]_G\ar[dl]\\
&{\bf Cat}
}
$$
\end{proposition}
\begin{proof}
Straightforward.
\end{proof}

\bibliographystyle{plain}
\bibliography{mybib}

\end{document}